\documentclass[12pt,a4]{amsart}
\usepackage[a4paper, left=28mm, right=28mm, top=28mm, bottom=34mm]{geometry}
\usepackage[all]{xy}
\usepackage{amsmath}
\usepackage{amssymb}
\usepackage{amsthm}
\usepackage{mathrsfs}
\theoremstyle{definition}
\newtheorem{thm}{Theorem}[section]

\newtheorem{cor}[thm]{Corollary}
\newtheorem{prop}[thm]{Proposition}

\theoremstyle{definition}

\newtheorem{rem}[thm]{Remark}

\newtheorem{defn}[thm]{Definition}

\newtheorem{ex}[thm]{Example}

\def\F{{\mathbb F}}

\def\Q{{\mathbb Q}}

\def\Z{{\mathbb Z}}

\def\O{{\mathscr O}}
\def\P{{\mathbb P}}

\def\Br{\mathop{\mathrm{Br}}\nolimits}

\def\Frac{\mathop{\mathrm{Frac}}\nolimits}

\def\Gal{\mathop{\mathrm{Gal}}\nolimits}

\def\inv{\mathop{\mathrm{inv}}\nolimits}

\def\Jac{\mathop{\mathrm{Jac}}\nolimits}

\def\GL{\mathop{\mathrm{GL}}\nolimits}

\def\Pic{\mathop{\mathrm{Pic}}\nolimits}

\def\deg{\mathop{\text{\rm deg}}\nolimits}

\def\divi{\mathop{\mathrm{div}}}

\newcommand{\transp}[1]{{}^{t}\!{#1}}

\begin{document}
\title[Local-global principle for symmetric determinantal representations]
{The local-global principle for symmetric determinantal representations of smooth plane curves
in characteristic two}
\author{Yasuhiro Ishitsuka}
\address{Department of Mathematics, Faculty of Science, Kyoto University, Kyoto 606-8502, Japan}
\email{yasu-ishi@math.kyoto-u.ac.jp}
%\thanks{}
\author{Tetsushi Ito}
\address{Department of Mathematics, Faculty of Science, Kyoto University, Kyoto 606-8502, Japan}
\email{tetsushi@math.kyoto-u.ac.jp}
%\thanks{}

\date{\today}
\subjclass[2010]{Primary 14H50; Secondary 11D41, 14F22, 14G17, 14K15, 14K30}
% 11D41  Higher degree equations; Fermat's equation
% 14F22  Brauer groups of schemes
% 14G17  Positive characteristic ground fields
% 14H50  Plane and space curves
% 14K15  Arithmetic ground fields
% 14K30  Picard schemes, higher Jacobians

\keywords{plane curve, determinantal representation, local-global principle, theta characteristic, characteristic two}

\maketitle

\begin{abstract}
We give an application of Mumford's theory of
canonical theta characteristics to a Diophantine problem in characteristic two.
We prove that a smooth plane curve over a global field of characteristic two
is defined by the determinant of a symmetric matrix with entries in linear forms
in three variables if and only if
such a symmetric determinantal representation exists everywhere locally.
It is a special feature in characteristic two because
analogous results are not true in other characteristics.
\end{abstract}

% \tableofcontents

\section{Introduction}

Let $C \subset \P^2$ be a smooth plane curve of degree $d \geq 1$
over a field $K$.
The plane curve $C$ is said to admit a {\em symmetric determinantal representation} over $K$
if there is a symmetric matrix $M$ of size $d$
with entries in $K$-linear forms in three variables $X,Y,Z$
such that $C$ is defined by the equation $\det(M)=0$.
Two symmetric determinantal representations $M,M'$ of $C$ are said to be
{\em equivalent} if $M' = \lambda \, \transp{S} M S$ for some $S \in \GL_d(K)$
and $\lambda \in K^{\times}$, where $\transp{S}$ is the transpose of $S$.
Studying symmetric determinantal representations is a classical topic
in algebraic geometry and linear algebra,
which goes back to Hesse's work on plane cubics and quartics;
\cite{Dixon}, \cite[Ch 4]{Dolgachev}.
Recently, arithmetic properties of symmetric determinantal representations
and related linear orbits are studied by several people;
\cite{Ho}, \cite{BhargavaGrossWang}, \cite{IshitsukaIto:Fermat}, \cite{IshitsukaIto:LocalGlobal}.

In this paper, we prove the local-global principle
for the existence of symmetric determinantal representations in characteristic two.

\begin{thm}[see Theorem \ref{MainTheoremGeneral}]
\label{MainTheorem}
Let $K$ be a global field of {\em characteristic two},
and $C \subset \P^2$ a smooth plane curve of degree $d \geq 1$ over $K$.
Then the following are equivalent:
\begin{enumerate}
\item $C$ admits a symmetric determinantal representation over $K$.
\item $C$ admits a symmetric determinantal representation over $K_v$ for every place $v$ of $K$.
\end{enumerate}
\end{thm}

In our previous paper \cite{IshitsukaIto:LocalGlobal},
we studied the local-global principle for the existence of symmetric determinantal
representations over a global field of characteristic $\neq 2$.
We proved the local-global principle for conics and cubics;
see \cite[Theorem 1.1]{IshitsukaIto:LocalGlobal}.
For quartics, we constructed examples failing
the local-global principle under mild conditions on the characteristic of the global field;
see \cite[Corollary 1.3]{IshitsukaIto:LocalGlobal}.
According to some group theoretic results in \cite[Section 6]{IshitsukaIto:LocalGlobal},
we expect that similar examples failing the local-global principle exist
in any degree $\geq 4$ in any characteristic $\neq 2$;
see \cite[Problem 1.6]{IshitsukaIto:LocalGlobal}.

In this paper,
we study the remaining case of {\em characteristic two}.
Rather interestingly, the story is completely different.
The local-global principle for the existence of symmetric determinantal
representations holds in any degree in characteristic two.
The key is Mumford's theory of {\em canonical theta characteristics}
on smooth projective curves which exist only in characteristic two.

The outline of this paper is as follows.
After recalling a relation between symmetric determinantal representations
and non-effective theta characteristics in
Section \ref{Section:SymmetricDeterminantalRepresentations},
we recall Mumford's theory of canonical theta characteristics
in Section \ref{Section:ThetaCharacteristicChar2}.
Then we examine the case of conics and cubics in Section \ref{Section:Example}.
For conics and cubics,
the existence of symmetric determinantal representations
is related to other (more familiar) Diophantine problems.
We also give some examples of symmetric determinantal representations
for conics and cubics.
Then the main theorem is proved in Section \ref{Section:ProofMainTheorem}.
Finally, in Section \ref{Section:LinearDeterminantalRepresentation},
we consider an analogous problem for {\em linear determinantal representations}
(i.e.\ the matrix $M$ is not assumed to be symmetric).
It turns out that, concerning the existence of linear determinantal representations,
the local-global principle does not hold even for cubics.
Hence Theorem \ref{MainTheorem} cannot be generalized to the linear case.

\subsection*{Notation}

An algebraic closure of a field $K$ is denoted by $\overline{K}$,
and a separable closure of $K$ is denoted by $K^{\mathrm{sep}}$.
A global field of characteristic two is a finite extension of $\F_2(T)$,
where $\F_2$ is the finite field with two elements and $T$ is an indeterminate.
For a place $v$ of $K$, the completion of $K$ at $v$ is denoted by $K_v$.

\section{Symmetric determinantal representations and theta characteristics}
\label{Section:SymmetricDeterminantalRepresentations}

In this section, let $K$ be an arbitrary field.

\begin{defn}
Let $C$ be a projective smooth geometrically connected curve over $K$.
A {\em theta characteristic} on $C$ is a line bundle $\mathscr{L}$
satisfying $\mathscr{L} \otimes \mathscr{L} \cong \Omega^1_{C}$,
where $\Omega^1_{C}$ is the canonical sheaf on $C$.
A theta characteristic $\mathscr{L}$ on $C$ is {\em effective} (resp.\ {\em non-effective}) 
if $H^0(C,\mathscr{L}) \neq 0$ (resp.\ $H^0(C,\mathscr{L}) = 0$).
\end{defn}

The following result is classical and well-known at least when
the base field is algebraically closed of characteristic zero; see \cite{Dixon}, \cite[Ch 4]{Dolgachev}.
In fact, it is valid over arbitrary fields.

\begin{prop}[Beauville]
\label{Proposition:ExistenceCriterion}
Let $C \subset \P^2$ be a smooth plane curve over $K$.
There is a natural bijection between the following sets:
\begin{itemize}
\item the set of equivalence classes of symmetric determinantal representations of $C$ over $K$, and
\item the set of isomorphism classes of non-effective theta characteristics on $C$ defined over $K$.
\end{itemize}
\end{prop}

\begin{proof}
See \cite[Proposition 4.2]{BeauvilleDeterminantal}.
See also
\cite[Proposition 2.2, Corollary 2.3]{IshitsukaIto:Fermat},
\cite[Theorem 2.2]{IshitsukaIto:LocalGlobal},
\end{proof}

\section{Canonical theta characteristics in characteristic two}
\label{Section:ThetaCharacteristicChar2}

In this section, let $K$ be a field of {\em characteristic two}.

In his foundational paper on theta characteristics,
Mumford observed the following ``strange'' fact; see \cite[p.191]{MumfordTheta}.
Let $C$ be a projective smooth geometrically connected curve over $K$,
and $K(C)$ the function field of $C$.
Take a rational function $f \in K(C)$ which is not a square in $\overline{K}(C)$.
An easy local calculation shows that the order of the divisor $\divi(df)$
at every geometric point $x \in C(\overline{K})$ is an even integer.
Hence $\divi(df)$ is always divisible by 2 as a divisor on $C \otimes_K \overline{K}$.
The associated line bundle
$\mathscr{L}_{\mathrm{can}} := \O_{C \otimes_K \overline{K}}\big( \frac{1}{2} \divi(df) \big)$
is called {\em the canonical theta characteristic} on $C \otimes_K \overline{K}$.
The isomorphism class of $\mathscr{L}_{\mathrm{can}}$ does not depend on the choice of $f$.

We are interested in non-effective theta characteristics
(cf.\ Proposition \ref{Proposition:ExistenceCriterion}).
Using the Cartier operator and logarithmic differential forms,
St\"ohr and Voloch proved the following beautiful results.

\begin{thm}[St\"ohr-Voloch]
\label{Theorem:ThetaCharacteristic}
\begin{enumerate}
\item Any theta characteristic on $C \otimes_K \overline{K}$ not isomorphic to the canonical theta characteristic is effective.
\item The canonical theta characteristic $\mathscr{L}_{\mathrm{can}}$
is non-effective if and only if the Jacobian variety $\Jac(C)$ is ordinary.
\end{enumerate}
\end{thm}

\begin{proof}
See \cite[\S 3]{StoehrVoloch}.
\end{proof}

Combining Proposition \ref{Proposition:ExistenceCriterion} and
Theorem \ref{Theorem:ThetaCharacteristic},
we have the following corollary.

\begin{cor}
\label{Corollary:SymmetricDeterminantalRepresentationExistence}
Let $K$ be a field of characteristic two,
and $C \subset \P^2$ a smooth plane curve of degree $d \geq 1$ over $K$.
\begin{enumerate}
\item There is at most one equivalence class of symmetric determinantal
representations of $C$ over $K$.
\item $C$ admits a symmetric determinantal representation over $\overline{K}$
if and only if $\Jac(C)$ is ordinary.
\item $C$ admits a symmetric determinantal representation over $K$
if and only if $\Jac(C)$ is ordinary and
the canonical theta characteristic $\mathscr{L}_{\mathrm{can}}$ is defined over $K$.
\end{enumerate}
\end{cor}

\begin{rem}
The canonical theta characteristic
$\mathscr{L}_{\mathrm{can}}$ is defined as a line bundle on
$C \otimes_K \overline{K}$, not on $C$.
Since ``dividing the divisor $\divi(df)$ by 2'' is problematic in characteristic two,
$\mathscr{L}_{\mathrm{can}}$ is not defined over $K$ in general.
For conics and cubics,
we can give necessary and sufficient conditions for
$\mathscr{L}_{\mathrm{can}}$ to be defined over $K$;
see Proposition \ref{Proposition:Conics} and Proposition \ref{Proposition:Cubics}.
\end{rem}

\begin{rem}
By Corollary \ref{Corollary:SymmetricDeterminantalRepresentationExistence},
in characteristic two,
smooth plane curves with non-ordinary Jacobian varieties
do not admit symmetric determinantal representations
even over an algebraic closure of the base field.
In contrast, in characteristic zero,
all plane curves (including singular ones)
admit symmetric determinantal representations over
an algebraic closure of the base field;
see \cite{Dixon}, \cite[Remark 4.4]{BeauvilleDeterminantal}.
\end{rem}

\section{Examples: conics and cubics}
\label{Section:Example}

In order to illustrate the problem of finding symmetric determinantal representations
in characteristic two, we examine the case of smooth plane conics and cubics in some detail.
It turns out that the existence of the canonical theta characteristic
is related to other Diophantine problems.
For analogous results in characteristic $\neq 2$,
see \cite[Section 4]{IshitsukaIto:LocalGlobal}.

\subsection{Conics}
\label{Subsection:SmoothConics}

\begin{prop}
\label{Proposition:Conics}
Let $K$ be a field of characteristic two,
and $C \subset \P^2$ a smooth plane conic over $K$.
The following are equivalent:
\begin{enumerate}
\item The canonical theta characteristic $\mathscr{L}_{\mathrm{can}}$ is defined over $K$.
\item $C$ admits a symmetric determinantal representation over $K$.
\item $C$ is isomorphic to $\P^1$ over $K$.
\item $C$ has a $K$-rational point.
\item $C$ has a $K$-rational divisor of odd degree.
\end{enumerate}
\end{prop}

\begin{proof}
The equivalences
(2) $\Leftrightarrow$ (3) $\Leftrightarrow$ (4) $\Leftrightarrow$ (5)
are proved in \cite[Proposition 4.1]{IshitsukaIto:LocalGlobal}.
Since the genus of $C$ is zero, $\Jac(C)$ is ordinary.
Hence the equivalence
(1) $\Leftrightarrow$ (2) follows from Corollary
\ref{Corollary:SymmetricDeterminantalRepresentationExistence} (3).
We can prove it directly without using the results of St\"ohr-Voloch
as follows.
(1) $\Rightarrow$ (5): If $\mathscr{L}_{\mathrm{can}}$ is defined over $K$,
we have $\deg \mathscr{L}_{\mathrm{can}} = \frac{1}{2} \deg \Omega^1_{C} = -1$.
(4) $\Rightarrow$ (1): For a $K$-rational point $P \in C(K)$,
the line bundle $\mathcal{L} = \O_C(-P)$
is the canonical theta characteristic defined over $K$
because it is a line bundle of degree $-1$;
see \cite[Proposition 3.3]{IshitsukaIto:LocalGlobal}.
\end{proof}

\begin{ex}
\label{Example:Conic}
Let $B$ be a quaternion division algebra over a field $K$ of characteristic two.
Such a $B$ exists when $K$ is a global or local field.
Let $C_B$ be the Severi-Brauer variety associated with $B$; see \cite[Ch X, \S 6]{SerreLocalFields}.
Since $C_B$ has no $K$-rational point,
$C_B$ does not admit a symmetric determinantal representation over $K$ by Proposition \ref{Proposition:Conics}.
The conic $C_B$ admits a symmetric determinantal representation over
a {\em separable} quadratic extension of $K$
because $C_B$ has a rational point over a separable quadratic extension by Bertini's theorem;
see \cite[Corollaire 6.11 (2)]{Jouanolou}.
(We can use Bertini's theorem because $K$ is infinite.
The Brauer group of a finite field is trivial; see \cite[Ch X, \S 7]{SerreLocalFields}.)
Note that $C_B$ also admits a symmetric determinantal representation over
a {\em purely inseparable} quadratic extension of $K$.
To see this, write the defining equation of $C_B \subset \P^2$ as
\[ a X^2 + b Y^2 + c Z^2 + d X Y + e Y Z + f X Z = 0 \quad (a,b,c,d,e,f \in K). \]
Since $C_B$ is smooth, at least one of $d,e,f$ is not zero.
We may assume $d \neq 0$ by changing coordinates.
Since $C_B$ has no $K$-rational point, $[1 : 0 : 0 ]$ does not lie on $C_B$.
Hence $a \neq 0$.
We put $t := a^{-1} (b f^2 + c d^2 + e f d)$.
By direct calculation, the point $[\,\sqrt{t} : f : d\,]$ lies on $C_B$.
By Proposition \ref{Proposition:Conics},
$C_B$ admits a symmetric determinantal representation over $K(\sqrt{t})$,
which is a purely inseparable quadratic extension of $K$.
\end{ex}

\begin{ex}
In this example, let $K$ be a field of {\em arbitrary characteristic}.
Any smooth plane conic $C \subset \P^2$ over $K$
admits at most one equivalence class of symmetric determinantal
representations over $K$
because $\Pic(C)$ has no non-trivial $2$-torsion element.
(In fact, we have $\Pic(C) \subset \Pic(C \otimes_K \overline{K}) \cong \Z$;
see \cite[Section 4.2]{IshitsukaIto:LocalGlobal}.)
When $C$ admits one, it is possible to calculate
a symmetric determinantal representation explicitly as follows.
Let
\[ a X^2 + b Y^2 + c Z^2 + d X Y + e Y Z + f X Z = 0 \quad (a,b,c,d,e,f \in K) \]
be the defining equation of $C$.
By Proposition \ref{Proposition:Conics},
$C$ has a $K$-rational point.
By a linear change of variables,
we may assume that $[1 : 0 : 0 ]$ is a $K$-rational point on $C$,
and the tangent line to $C$ at $P$ is $(Z = 0)$.
Then we have $a = d = 0$ and $f \neq 0$.
Since the conic $C$ intersects with the line $(Z = 0)$ at $[1 : 0 : 0]$
with multiplicity two,
$[0 : 1 : 0 ]$ does {\em not} lie on $C$. Hence we have $b \neq 0$.
A symmetric determinantal representation of $C$ over $K$
is given by the symmetric matrix
\[ M = \begin{pmatrix}
Z  & bY \\
bY & - bfX - beY - bcZ
\end{pmatrix} \]
because $\det(M) = -b \cdot ( b Y^2 + c Z^2 + e Y Z + f X Z )$.
\end{ex}

\subsection{Cubics}

By Corollary \ref{Corollary:SymmetricDeterminantalRepresentationExistence} (3),
a smooth plane cubic in characteristic two
does not admit a symmetric determinantal representation
when its Jacobian variety is supersingular.
We shall only consider the ordinary case.

\begin{prop}
\label{Proposition:Cubics}
Let $K$ be a field of characteristic two,
and $C \subset \P^2$ a smooth plane cubic over $K$.
Assume that the Jacobian variety $\Jac(C)$ is ordinary.
The following are equivalent:
\begin{enumerate}
\item The canonical theta characteristic $\mathscr{L}_{\mathrm{can}}$ is defined over $K$.
\item $C$ admits a symmetric determinantal representation over $K$.
\item $\Jac(C)$ has a non-trivial $K$-rational 2-torsion point.
\end{enumerate}
\end{prop}

\begin{proof}
(1) $\Leftrightarrow$ (2) follows from Corollary \ref{Corollary:SymmetricDeterminantalRepresentationExistence} (3).
(2) $\Leftrightarrow$ (3) is proved in \cite[Proposition 4.2]{IshitsukaIto:LocalGlobal}.
\end{proof}

\begin{rem}
Since $\Jac(C)$ is an ordinary elliptic curve in characteristic two,
its $j$-invariant is not zero.
We may assume that the Weierstrass equation of $\Jac(C)$ is given by the following form:
\[
Y^2 Z + XYZ = X^3 + a_2 X^2 Z + a_6 Z^3 \quad (a_2,a_6 \in K,\ a_6 \neq 0).
\]
See \cite[Appendix A]{Silverman}.
A direct calculation using \cite[Ch III, Group Law Algorithm 2.3]{Silverman}
shows that the non-trivial $\overline{K}$-rational 2-torsion point on $\Jac(C)$ is
$[0 : \sqrt{a_6} : 1]$.
Hence the conditions in Proposition \ref{Proposition:Cubics}
are satisfied if and only if $\sqrt{a_6} \in K$.
Therefore,
if $C$ does not admit a symmetric determinantal representation over $K$,
it does not admit one over $K^{\mathrm{sep}}$.
But, it admits one over $K(\sqrt{a_6})$, which is {\em purely inseparable} over $K$.
Compare these results with the results in Example \ref{Example:Conic}.
\end{rem}

\begin{rem}
In principle, it should be possible
to calculate symmetric determinantal representations of smooth plane cubics
corresponding to non-trivial $K$-rational $2$-torsion points.
An explicit calculation for general cubics
seems a computationally hard problem.
The situation becomes much simpler when $C$ has a $K$-rational point.
In that case, there is no obstruction accounted by the Brauer group $\Br(K)$
(or the relative Brauer group $\Br(C/K)$).
In \cite{Ishitsuka:FiniteField}, \cite{Ishitsuka:JSIAMLetters},
the first author gave an algorithm
to calculate linear (i.e.\ not necessarily symmetric) determinantal representations of
smooth plane cubics with rational points.
\end{rem}

\begin{ex}
The following example over $\F_2$ is calculated in \cite[Table 6]{Ishitsuka:FiniteField}:
\[
M =
\begin{pmatrix}
Y & 0 & X \\
0 & Z & Y \\
X & Y & X+Y+Z
\end{pmatrix}
\]
The cubic
\[ C \colon \det(M) = X^2 Z + XYZ + Y^3 + Y^2 Z + Y Z^2 = 0 \]
is a unique smooth plane cubic over $\F_2$, up to projective equivalence,
admitting only one equivalence class of linear (i.e.\ not necessarily symmetric)
determinantal representations over $\F_2$.
Similar cubics exist over $\F_3, \F_4, \F_5$.
But there are no such cubics over $\F_q \ (q \geq 7)$.
See \cite{Ishitsuka:FiniteField}, \cite{Ishitsuka:JSIAMLetters} for details.
\end{ex}

\begin{ex}
Here we give examples of symmetric determinantal representations
of twisted (or generalized) Hesse cubics
\[ C_{a,b,c,m} : a X^3 + b Y^3 + c Z^3 + m XYZ = 0 \quad
   (a,b,c,m \in K,\ abc(m^3 + 27 abc) \neq 0) \]
over a field $K$ of {\em characteristic two}.
These curves were studied by Desboves in 19th century;
\cite{Desboves}, \cite[p.130]{CasselsEllipticCurve}.
A formula given by Artin, Rodriguez-Villegas and Tate
(\cite[(1.5), (1.6)]{ArtinRodriguezVillegasTate})
shows that the Jacobian variety $\Jac(C_{a,b,c,m})$
is defined by the following Weierstrass equation:
\[ \Jac(C_{a,b,c,m}) : Y^2 Z + m XYZ + 9abc YZ^2 = X^3 + (-27 a^2 b^2 c^2 + m^3 abc)Z^3. \]
Then $\Jac(C_{a,b,c,m})$ is ordinary if and only if $m \neq 0$.
Using \cite[Ch III, Group Law Algorithm 2.3]{Silverman},
it is not difficult to show that
$\Jac(C_{a,b,c,m})$ has a non-trivial $K$-rational $2$-torsion point if and only if
$\sqrt{m^{-1} abc} \in K$.
When $m^{-1} abc$ is a square in $K$, the unique equivalence class of
symmetric determinantal representations of $C_{a,b,c,m}$ over $K$
is represented by
\[
M_{a,b,c,m} =
\begin{pmatrix}
aX & \sqrt{m^{-1} abc} \, Z & \sqrt{m^{-1} abc} \, Y \\
\sqrt{m^{-1} abc} \, Z & bY & \sqrt{m^{-1} abc} \, X \\
\sqrt{m^{-1} abc} \, Y & \sqrt{m^{-1} abc} \, X & cZ
\end{pmatrix}.
\]
In fact, we can directly check
\[ \det(M_{a,b,c,m}) = m^{-1} abc \, ( a X^3 + b Y^3 + c Z^3 + m XYZ). \]
Details on the calculation of
symmetric and linear determinantal representations of twisted Hesse cubics
will appear elsewhere.
Some examples over $\Q$ are given in \cite[Section 6]{Ishitsuka:JSIAMLetters}
for $m = 0$ (i.e.\ the case of ``twisted Fermat cubics'').
\end{ex}

\section{Proof of the main theorem}
\label{Section:ProofMainTheorem}

The following Theorem \ref{MainTheoremGeneral} is stronger than Theorem \ref{MainTheorem}.

For conics and cubics, Theorem \ref{MainTheoremGeneral} can be proved
directly using Proposition \ref{Proposition:Conics} and Proposition \ref{Proposition:Cubics}.
However, it seems difficult to generalize these propositions
because, as you see from the proof of Theorem \ref{MainTheoremGeneral} below,
giving such results amount to giving a defining equation of
the ``locus of canonical theta characteristics'' in the Picard scheme
and a splitting of a 2-torsion element in the Brauer group.
It seems an infeasible task in higher degree.
Instead, our approach to prove Theorem \ref{MainTheoremGeneral}
is to use Greenberg's approximation theorem
(\cite[Theorem 1]{GreenbergRationalPoints})
and some cohomological results on Brauer groups.

\begin{thm}
\label{MainTheoremGeneral}
Let $K$ be a global field of characteristic two,
and $C \subset \P^2$ a smooth plane curve of degree $d \geq 1$.
Assume that
$C$ admits a symmetric determinantal representation over $K_{v_0}$ for a place $v_0$ of $K$.
Moreover, assume that {\em at least one} of the following conditions is satisfied:
\begin{itemize}
\item $C$ has a $K$-rational point,
\item $d$ is odd, or
\item there is a place $w$ of $K$ such that
$C$ admits a symmetric determinantal representation over $K_{v}$ for every place $v \neq w$ of $K$.
\end{itemize}
Then $C$ admits a symmetric determinantal representation over $K$.
\end{thm}

\begin{proof}
Since $C$ admits a symmetric determinantal representation over $K_{v_0}$,
by Corollary \ref{Corollary:SymmetricDeterminantalRepresentationExistence} (3),
$\Jac(C)$ is ordinary and $\mathscr{L}_{\mathrm{can}}$ is defined over $K_{v_0}$.
Hence it is enough to prove that $\mathscr{L}_{\mathrm{can}}$ is defined over $K$.

We shall first prove that
$\mathscr{L}_{\mathrm{can}}$ gives a $K$-rational point on the relative Picard scheme $\Pic_{C/K}$
(\cite[Ch 8]{BoschLuetkebohmertRaynaud}, \cite[Section 3]{IshitsukaIto:LocalGlobal}).
The canonical sheaf $\Omega^1_{C}$ gives
a $K$-rational point $[\Omega^1_{C}] \in \Pic_{C/K}(K)$.
Let $T \subset \Pic_{C/K}$ be the fiber over $[\Omega^1_{C}]$
under the multiplication-by-2 map
\[ [2] \colon \Pic_{C/K} \longrightarrow \Pic_{C/K}. \]
The finite $K$-scheme $T$ is called the {\em theta characteristic torsor} (\cite{PoonenRains}).
Since $\mathscr{L}_{\mathrm{can}}$ is defined over $K_{v_0}$,
we have a $K_{v_0}$-rational point $[\mathscr{L}_{\mathrm{can}}] \in T(K_{v_0})$.
Let $\O_{v_0}$ be the ring of integers of $K_{v_0}$.
Then $R := \O_{v_0} \cap K$
is an excellent discrete valuation ring whose completion is $\O_{v_0}$
because it is the localization of a finitely generated $\F_2$-algebra;
see \cite[(7.8.3) (ii),(iii)]{EGA4-2}.
Let $R^{\mathrm{h}}$ be the Henselization of $R$.
Let $\mathfrak{T}$ be the integral closure of $R$ in $T$.
Then $\mathfrak{T}$ is a finite flat $R$-scheme with generic fiber $T$.
By the valuative criterion of properness,
the $K_{v_0}$-rational point $[\mathscr{L}_{\mathrm{can}}] \in T(K_{v_0})$
extends to an $\O_{v_0}$-rational point $x \in \mathfrak{T}(\O_{v_0})$.

By Greenberg's approximation theorem (\cite[Theorem 1]{GreenbergRationalPoints}),
for any $N \geq 1$,
there is an $R^{\mathrm{h}}$-rational point $y_N \in \mathfrak{T}(R^{\mathrm{h}})$
such that $x, y_N$ have the same image in $\mathfrak{T}(R/m^N)$,
where $m$ is the maximal ideal of $R$.
Since $\mathfrak{T}$ is a finite $R$-scheme,
if we take $N$ to be sufficiently large,
the image of $y_N$ in $\mathfrak{T}(\O_{v_0})$ coincides with $x$.
Therefore, $x$ is an $R^{\mathrm{h}}$-rational point,
and $[\mathscr{L}_{\mathrm{can}}]$ is a $(\Frac R^{\mathrm{h}})$-rational point.

Since the extension $(\Frac R^{\mathrm{h}})/K$ is separable and algebraic,
$[\mathscr{L}_{\mathrm{can}}]$ is a $K^{\mathrm{sep}}$-rational point.
Moreover, it is easy to see from the construction of $\mathscr{L}_{\mathrm{can}}$ that
$[\mathscr{L}_{\mathrm{can}}] \in \Pic_{C/K}(K^{\mathrm{sep}})$
is fixed by the action of $\Gal(K^{\mathrm{sep}}/K)$.
Hence it is a $K$-rational point, i.e.\ we have
\[ [\mathscr{L}_{\mathrm{can}}] \in \Pic_{C/K}(K). \]

We have the exact sequence of low-degree terms coming from the Leray spectral sequence
(cf.\ \cite[Ch 8]{BoschLuetkebohmertRaynaud}, \cite[Section 3]{IshitsukaIto:LocalGlobal}):
\[
\xymatrix{
0 \ar[r] & \Pic(C)
  \ar[r] & \Pic_{C/K}(K)
  \ar[r] & \Br(K)
  \ar[r] & \Br(C).
}
\]
It remains to prove that
$[\mathscr{L}_{\mathrm{can}}] \in \Pic_{C/K}(K)$
comes from a line bundle on $C$.
We shall prove it using one of the assumptions of the theorem.
Let $\alpha \in \Br(K)$ be the image of $[\mathscr{L}_{\mathrm{can}}]$
in the Brauer group $\Br(K)$.

\begin{itemize}
\item If $C$ has a $K$-rational point, we have $\Pic(C) \cong \Pic_{C/K}(K)$
(\cite[Ch 8, Section 1, Proposition 4]{BoschLuetkebohmertRaynaud}).
Hence $[\mathscr{L}_{\mathrm{can}}]$ comes from a line bundle on $C$.

\item If $d$ is odd, there is a finite extension $L/K$ of odd degree such that
$C$ has an $L$-rational point.
By a standard cohomological argument using $\mathrm{Res}$ and $\mathrm{Cor}$
(cf.\ \cite[Ch VII, Proposition 6]{SerreLocalFields}),
we see that $\alpha$ is killed by $[L:K]$.
On the other hand, since
$2 [\mathscr{L}_{\mathrm{can}}] = [\mathscr{L}_{\mathrm{can}} \otimes \mathscr{L}_{\mathrm{can}}] = [\Omega^1_{C}]$
comes from the canonical sheaf on $C$, we see that $\alpha$ is also killed by $2$.
Hence $\alpha$ is trivial,
and $[\mathscr{L}_{\mathrm{can}}]$ comes from a line bundle on $C$.

\item If $C$ admits a symmetric determinantal representation over $K_{v}$
for every place $v \neq w$ of $K$,
the element $\alpha$ lies in the kernel of the homomorphism
\[
\psi \colon \Br(K)\ \longrightarrow \bigoplus_{v\,:\,\text{place of}\ K,\ v \neq w} \Br(K_v).
\]
The exact sequence
\[
0 \ \longrightarrow \ \Br(K)\ \longrightarrow \bigoplus_{v\,:\,\text{place of}\ K} \Br(K_v)
\ \overset{\sum_v \inv_v}{\longrightarrow} \ \Q/\Z \ \longrightarrow \ 0
\]
in global class field theory (\cite[Theorem 8.1.17]{NeukirchSchmidtWingberg}, \cite[\S 9, \S 10]{TateGCFT})
shows that $\psi$ is injective.
Hence $\alpha$ is trivial,
and $[\mathscr{L}_{\mathrm{can}}]$ comes from a line bundle on $C$.
\end{itemize}

The proof of Theorem \ref{MainTheoremGeneral} is complete.
\end{proof}

\begin{rem}
In higher degree,
we do not know how to explicitly calculate a symmetric determinantal representation
of a smooth plane curve once we know it exists.
By \cite[Proposition 4.2]{BeauvilleDeterminantal},
it amounts to giving an explicit minimal free resolution of the graded
module associated with $\iota_{\ast} \mathscr{L}_{\mathrm{can}}$,
where $\iota \colon C \hookrightarrow \P^2$ is a fixed embedding.
\end{rem}

\section{The case of linear determinantal representations}
\label{Section:LinearDeterminantalRepresentation}

It is a natural question to study determinantal representations
which are not necessarily symmetric.
A smooth plane curve $C \subset \P^2$ of degree $d$ over a field $K$
is said to admit a {\em linear determinantal representation} over $K$
if there is a square (not necessarily symmetric) matrix $M$ of size $d$
with entries in $K$-linear forms in three variables $X,Y,Z$
such that $\det(M)=0$ is the defining equation of $C$.

It turns out that,
concerning the existence of linear determinantal representations,
the local-global principle holds for $d = 2$;
see \cite[Corollary 4.2]{Ishitsuka:PositiveProportion}.
But, it does {\em not} hold for $d = 3$.
The following result is proved by the first author.

\begin{prop}
\label{Proposition:Linear}
Let $K$ be a global field of arbitrary characteristic,
and $C \subset \P^2$ a smooth plane cubic over $K$.
Assume that the Mordell-Weil group $\Jac(C)(K)$ is trivial.
Then $C$ does not admit a linear determinantal representation over $K$,
whereas $C$ admits a linear determinantal representation over $K_v$ for every place $v$ of $K$.
\end{prop}

\begin{proof}
See \cite[Corollary 4.2]{Ishitsuka:PositiveProportion}.
\end{proof}

It is a simple matter to find smooth plane cubics
whose Jacobian varieties have trivial Mordell-Weil group.
For example, the following elliptic curves
\begin{align*}
E_1 &: Y^2 Z + T^3 Y Z^2 = X^3 + T^5 Z^3 \\
E_2 &: Y^2 Z + T XYZ = X^3 + T^5 Z^3
\end{align*}
over $\F_2(T)$ have trivial Mordell-Weil group.
These are taken from \cite[Table 3 in p.182]{Ito:Hiroyuki}.
By Proposition \ref{Proposition:Linear},
$E_1$ and $E_2$ do not admit linear determinantal representations over $\F_2(T)$,
but they admit linear determinantal representations
everywhere locally.

Since Proposition \ref{Proposition:Linear}
is valid over any global field,
we may find similar examples in other characteristics.
In \cite{Ishitsuka:PositiveProportion},
the proportion of smooth plane cubics over $\Q$
failing the local-global principle
for the existence of linear determinantal representations
is studied.

\subsection*{Acknowledgements}

When we studied the canonical theta characteristics on
algebraic curves in characteristic two,
the posts at \texttt{mathoverflow} were very helpful
(\textit{Effectiveness of the distinguished theta characteristic in characteristic 2} \\
(\texttt{http://mathoverflow.net/questions/139698/})).
We would like to take this occasion to thank
the contributors to this wonderful website.
The work of the first author was supported by JSPS KAKENHI Grant Number 13J01450 and 16K17572.
The work of the second author 
was supported by JSPS KAKENHI Grant Number 20674001 and 26800013.

\end{document}